\documentclass{amsart}
\usepackage[foot]{amsaddr}
\usepackage{graphicx} 
\usepackage[paper=a4paper, top=2.2cm, 
	bottom=1.6cm, 
	left=2.5cm, 
	right=2.5cm, 
	headheight=0.75cm, 
	footskip=1cm, 
	headsep=0.5cm]{geometry}

\usepackage[hidelinks]{hyperref}

\usepackage[utf8]{inputenc}
\usepackage[
  backend=biber,
  style=numeric,
  sorting=none,      
  giveninits=true,   
  doi=true,          
  url=false,
  isbn=false,
  eprint=false
]{biblatex}

\AtEveryBibitem{%
  \clearfield{issn}%
  \clearfield{eissn}%
}

\renewbibmacro*{doi+eprint+url}{}

\DeclareFieldFormat{doi}{\href{https://doi.org/#1}{#1}}
\renewbibmacro*{finentry}{%
  \iffieldundef{doi}
    {}
    {\setunit{\adddot\space}%
     \printfield{doi}}%
  \finentry
}

\usepackage{verbatim}
\usepackage{xcolor}
\usepackage{float}
\usepackage{amsmath}
\usepackage{amssymb}
\usepackage{mathtools}
\usepackage{esint}
\usepackage{varioref}
\usepackage{amsthm}
\usepackage{comment}
\usepackage{enumitem}
\usepackage{ifthen}
\usepackage[capitalize,noabbrev]{cleveref}

\usepackage{cancel}

\usepackage[toc,page]{appendix}

\usepackage{mathrsfs}

\newtheorem{theorem}{Theorem}[section]

\newtheorem{lemma}[theorem]{Lemma}
\newtheorem{proposition}[theorem]{Proposition}
\newtheorem{assumption}[theorem]{Assumption}
\theoremstyle{definition}
\newtheorem{definition}[theorem]{Definition}

\newtheorem{remark}[theorem]{Remark}

\renewcommand{\d}{\mathrm{d}}
\renewcommand{\H}{\mathrm{H}}
\renewcommand{\L}{\mathrm{L}}

\newcommand{\opt}{\mathrm{opt}}

\newcommand{\re}{\mathrm{Re}\,}

\newcommand{\Lp}[1][p]{\L^{#1}}

\newcommand{\R}{\mathbb{R}}
\newcommand{\C}{\mathbb{C}}

\newcommand{\N}{\mathbb{N}}

\newcommand{\citel}[2]{\cite[#2]{#1}}
\newcommand{\eq}[1]{\begin{align*}#1 \end{align*}}

\newcommand{\pmatsmall}[1]{\begin{bsmallmatrix}#1 \end{bsmallmatrix}}

\newcommand*{\Iprod}[3][default]{\ifthenelse{\equal{#1}{default}}{\left\langle#2,#3\right\rangle}{\ldelim{#1}{\langle}#2,#3\rdelim{#1}{\rangle}}}

\usepackage{pgfplots}
\pgfplotsset{compat=newest}
\usepackage{graphicx}
\usepackage{tikz}
\usepackage{overpic}
\usepackage{stanli}
\usetikzlibrary{arrows.meta,positioning,intersections,shapes,arrows}

\usepackage{amsmath,amssymb,amsfonts,mathtools,bm}

\usepackage[inline]{asymptote}

\usepackage{times}
\usepackage{xcolor}
\usepackage{enumitem}

\DeclarePairedDelimiterX{\setdef}[2]{\{}{\}}{#1\,\delimsize\vert\,\mathopen{} #2}
\DeclarePairedDelimiterX{\scprod}[2]{\langle}{\rangle}{#1,#2}
\DeclarePairedDelimiterX{\dualprod}[2]{\langle}{\rangle}{#1,#2}
\DeclarePairedDelimiterX{\sdprod}[2]{\llangle}{\rrangle}{#1,#2} 
\newcommand{\dd}{\mathrm{d}} \newcommand{\dx}[1][x]{\mathop{\dd#1}}

\newcommand{\Lb}{\mathcal{L}_{\mathrm{b}}} 

\newcommand{\adjunsymb}{\ast} 

\newcommand{\adjun}[1][1]{%
  \setcounter{i}{1}%
  \striche={\adjunsymb}%
  \loop%
  \ifnum\value{i}<#1%
  \striche=\expandafter{\the\expandafter\striche\adjunsymb}%
  \stepcounter{i}%
  \repeat%
  ^{\the\striche}%
}
\newcommand{\X}{\mathcal{X}}
\newcommand{\U}{\mathcal{U}}
\newcommand{\Y}{\mathcal{Y}}
\newcommand{\W}{\mathcal{W}}

\renewcommand{\d}{\mathrm{d}}

\renewcommand{\H}{\mathrm{H}}
\renewcommand{\L}{\mathrm{L}}

\newcommand{\dom}{\operatorname{dom}}

\newcommand{\pvek}[2]
{
   \begin{pmatrix}
      #1\\
      #2
   \end{pmatrix}
}
\newcommand{\sbvek}[2]{\left[\begin{smallmatrix}#1\\#2\end{smallmatrix}\right]}
\newcommand{\spvek}[2]{\left(\begin{smallmatrix}#1\\#2\end{smallmatrix}\right)}

\usepackage{graphicx}          
\providecommand{\keywords}[1]
{
  \small	
  \textbf{\textit{Keywords---}} #1
}

\title{Optimal control of infinite-dimensional dissipative systems}

\author[A.\ Hastir]{Anthony Hastir}
\address[A.\ Hastir]{Department of Mathematics and Namur Institute for Complex Systems (naXys), University of Namur, Rue de Bruxelles 61, 5000 Namur, Belgium}
\email{anthony.hastir@unamur.be}

\author[T.\ Reis]{Timo Reis}
\address[T.\ Reis]{Institut für Mathematik, Technische Universität Ilmenau, Weimarer Straße 25, 98693 Ilmenau, Germany}
\email{timo.reis@tu-ilmenau.de}

\begin{document}



\begin{abstract}
We study the linear-quadratic optimal control problem for infinite-dimensional dissipative systems with possibly indefinite cost functional. Under the assumption that a storage function exists, we show that this indefinite optimal control problem is equivalent to a linear–quadratic optimal control problem with a nonnegative cost functional. We establish the relationship between the corresponding value functions and present the associated operator Lur'e equation. Finally, we illustrate our results with several examples.
\end{abstract}

\maketitle
\section{Introduction}
We consider an optimal control problem that consists in minimizing the infinite-horizon cost functional
\begin{equation}
J(u(\cdot),y(\cdot)) = \int_0^\infty
s(y(\tau),u(\tau))
\d\tau,\qquad
s(y(\tau),u(\tau)):=\left\langle
\spvek{y(\tau)}{u(\tau)},
\pmatsmall{Q & S\\[1mm] S^* & R}\spvek{y(\tau)}{u(\tau)}
\right\rangle_{\Y\times\U},
\label{eq:CostIntro}
\end{equation}
subject to a linear input–state–output system (to be specified later) with input $u(t)\in\U$, output $y(t)\in\Y$, and state $x(t)\in\X$ satisfying $x(0)=x_0\in\X$ and $\lim_{t\to\infty}x(t)=0$. Here, $\U$, $\Y$ and $\X$ are Hilbert spaces. In \eqref{eq:CostIntro}, the operators $Q:\Y\to\Y$ and $R:\U\to\U$ are bounded and self-adjoint, while $S:\U\to\Y$ is bounded.

Even though linear-quadratic optimal control problems admit a powerful treatment via methods from convex analysis \cite{EkelandTemam1999}, the situation becomes substantially more delicate in the infinite-dimensional setting. In particular, three structural features may significantly complicate the analysis:
\begin{enumerate}[label=(\roman*),ref=(\roman*)]
\item unbounded input and output operators,
\item singular cost functionals,
\item indefinite cost functionals.
\end{enumerate}

\medskip
\noindent
\textbf{Unbounded input and output operators.}
Unbounded control or observation operators typically arise in boundary control and observation of partial differential equations. In this case, one must work with suitable extrapolation spaces of the state space, and the characterization of linear-quadratic optimal control becomes considerably more involved than in the case of bounded (i.e., distributed) control and observation operators. This setting was investigated for weakly regular systems in \cite{WeissWeiss}, where associated operator equations (which generalize algebraic Riccati equations) were introduced. Subsequently, well-posed linear systems were treated in \cite{Staffans_LQ_1998}. More recently, an alternative approach based on system nodes has been developed; see \cite{Opmeer_Staffans,Opmeer2014}. There, an operator Lur'e equation (which is therein referred to as the \emph{operator-node Riccati equation}) is formulated, whose main advantage is that it avoids any extension of unbounded operators. By contrast, for bounded input and output operators the infinite-dimensional linear-quadratic optimal control problem is, at least conceptually, not significantly more difficult than its finite-dimensional counterpart; see, for instance, \cite{CurtainPritchard, CurtainZwart2020, CallierWinkin1990} and the classical finite-dimensional reference \cite{Kwakernaak_Sivan}.

\medskip
\noindent
\textbf{Singular problems.}
Singularity refers to the situation in which the cost functional $J(u(\cdot),0)$ fails to be coercive over the set of all admissible system trajectories. This typically, though not exclusively, arises when the operator $R$ is singular. In such cases, minimizing sequences of controls may exist, while an actual minimizer need not; even if it does, it may fail to be unique. This phenomenon already occurs in finite-dimensional settings \cite{geerts1994linearquadratic}, where one is naturally led to consider Lur’e equations, that is, matrix inequalities generalizing the algebraic Riccati equation.
Singular linear–quadratic problems in infinite-dimensional spaces were subsequently studied in \cite{Curtain2001LOI}, where an infinite-dimensional analogue of Lur’e equations was developed.

\medskip
\noindent
\textbf{Indefinite problems.}
By indefinite problems we mean that the block operator $\pmatsmall{Q & S\\ S^* & R}$ defining the cost functional is not positive semidefinite. 
In general, the infimum of the cost functional is then equal to $-\infty$. However, for the supply rate corresponding to $Q=R=0$ and $S=I$, passivity guarantees that the cost functional is bounded from below and thus has a finite infimum. Nevertheless, indefinite optimal control problems are intrinsically more challenging, as they no longer correspond to norm minimization problems. Despite these difficulties, they are of considerable practical relevance, since they naturally arise, for example, in energy minimization problems; see \cite{Faulwasser2022PHDAE}.

In this manuscript, we address infinite-horizon linear-quadratic optimal control problems in which unbounded input and output operators, singular costs, and indefiniteness may occur simultaneously. Our main contribution is a transformation technique that reduces an indefinite problem to an equivalent linear-quadratic optimal control problem with a nonnegative cost functional. In this setting, the functions $u$ and $y$ in \eqref{eq:CostIntro} denote, respectively, the input and output of a system of the form
\begin{equation}
    \spvek{\dot{x}(t)}{y(t)} = \pmatsmall{A\& B\\[1mm] C\& D}\spvek{x(t)}{u(t)},
    \label{eq:SysNode}
\end{equation}
where $\pmatsmall{A\& B\\ C\& D
}$ is a~so-called {\em system node} on the triple $(\U,\X,\Y)$ of Hilbert spaces. A~precise definition of this system class is recalled in Section \ref{sec:SysNode}. The later formalism may be surprising at first glance, notably because of the symbols ``$\&$" in \eqref{eq:SysNode}, but it provides an elegant way to describe systems with both boundary and distributed control in a single approach, see \cite{Sta05book}. Roughly speaking, $\pmatsmall{A\& B\\ C\& D}$ is the natural extension of the ``$ABCD$-notation'' for finite-dimensional linear systems. The symbol ``$\&$" stresses out that the domain of $A\& B$ and $C\& D$ is not necessarily a Cartesian product of subspaces of $\X$ and $\U$. 

In our developments, we assume that \eqref{eq:SysNode} is dissipative with respect to the supply rate $s$, see \eqref{eq:CostIntro}, and the storage function $\mathcal{S}:\X\supset\dom(\mathcal{S})\to\R$ in the sense of \cite[Def.~3.1]{Reis2025_Dissip}, which roughly speaking extends the notion of dissipativity introduced in the seminal works of Willems, see~\cite{Willems_Part1, Willems_Part2}. This concept of dissipativity is rather general and it incorporates for instance the notion of (impedance or scattering) passive systems whose properties are studied extensively in \cite{Sta02}. 

An important aspect of this manuscript is that we relax the nonnegativity assumption on the cost functional \eqref{eq:CostIntro}, which is usually required in the standard treatment of linear-quadratic optimal control problems. This relaxation leads to the consideration of indefinite cost functionals. As will be shown, our approach is based on transforming the resulting indefinite optimal control problem into an linear-quadratic optimal control problem with a semidefinite cost functional.

In this way, our work generalizes earlier results on optimal control for dissipative systems, such as those presented in \cite{HasJac_25} and \cite{ReisSchaller2025LQOC}. In \cite{HasJac_25}, linear-quadratic optimal control for impedance- or scattering-passive systems is studied in the framework of system nodes, under the structural assumption that the weighting operators in \eqref{eq:CostIntro} satisfy $Q = I$ and $S = 0$. In \cite{ReisSchaller2025LQOC}, indefinite cost functionals are considered in the context of port-Hamiltonian systems; however, the analysis is restricted to a finite time horizon and additionally includes a terminal-state penalization.

The manuscript is organized as follows: the concept of system nodes is introduced in Section~\ref{sec:SysNode}. In Section~\ref{sec:OptControl}, we show how to convert our indefinite optimal control problem into a definite one. Moreover, we show the relation between the different value functions and we give the appropriate underlying Lur'e equation. We apply our results to several examples in Section~\ref{sec:Examples}.

\section{System nodes}\label{sec:SysNode}
In this section, we present some background on system nodes that we will need in the remaining of this paper. 
Throughout this work, $\Lb(\X,\Y)$ denotes the set of linear and bounded operators from $\X$ to $\Y$. The abbreviation $\Lb(\X)$ is used to denote $\Lb(\X,\X)$.
We start with the definition of a system node.
\begin{definition}[\textup{\citel{Sta05book}{Def.~4.7.2}}]
\label{def:SysNode} 
Let $\X$, $\U$, and $\Y$ be Hilbert spaces.
A closed operator 
\eq{
S:=\pmatsmall{A\&B\\[1mm] C\&D} : \dom(S) \subset \X\times \U \to \X\times \Y
}
is called a \emph{system node} on the spaces $(\U,\X,\Y)$ if it has the following properties.
\begin{itemize}
\item The operator $A: \dom(A)\subset \X\to \X$ with $Ax =A\&B \spvek{x}{0}$ for $x\in \dom(A)=\setdef{x\in \X}{\spvek{x}0\in \dom(S)}$ generates a strongly continuous semigroup on $\X$.
\item For all $u\in\U$, there exists some $x\in\X$, such that $\spvek{x}{u}\in\dom(A\&B)$.
\item $\dom(S)=\dom(A\&B)$.
\end{itemize}
\end{definition}
The space $\X_{-1}$ is the completion of $\X$ with respect to the norm
$\|x\|_{-1} := \|(\beta-A)^{-1}x\|_{\X}$ for some $\beta\in\rho(A)$, where $\rho(A)$ denotes the resolvent set of $A$.
The operator $A\& B$ extends uniquely to a bounded operator
$[A_{-1},\,B]\in \Lb(\X\times\U,\X_{-1})$.
Moreover, the conditions in Definition~\ref{def:SysNode} ensure that $C\& D\in \Lb(\dom(S),\Y)$, where $\dom(S)$ is endowed with the graph norm of $S$.
The output operator $C\in \Lb(\dom(A),\Y)$ associated with $S$ is then defined by
$Cx := C\& D \spvek{x}{0}$ for all $x\in \dom(A)$.
We now introduce the notions of classical and generalized trajectories for system nodes; see \citel{TucWei14}{Def.~4.2}.
\begin{definition}\label{def:classicalSol_SN}
Let $S$ be a system node on the spaces $(\U,\X,\Y)$ with domain $\dom(S)$. A triple $(x(\cdot),u(\cdot),y(\cdot))$ is called a {\em classical trajectory} of \eqref{eq:SysNode} on $\R_{\ge0}$ if 
\begin{itemize}
    \item $x(\cdot)\in C^1(\R_{\geq 0};\X)$, $u(\cdot)\in C(\R_{\geq 0};\U)$ and $y(\cdot)\in C(\R_{\geq 0};\Y)$;
    \item $\spvek{x(t)}{u(t)}\in \dom(S)$ for all $t\geq 0$;
    \item \eqref{eq:SysNode} holds for all $t\geq 0$.
\end{itemize}
\end{definition}
According to \citel{TucWei14}{Sec.~4}, it follows from Definition \ref{def:classicalSol_SN} that every classical trajectory of \eqref{eq:SysNode} on $\R_{\geq 0}$ also satisfies $\spvek{x(\cdot)}{u(\cdot)}\in C(\R_{\geq 0};\dom(S))$.
\begin{definition}\label{def:genSol_SN}
Let $S$ be a system node on the spaces $(\U,\X,\Y)$ with domain $\dom(S)$. A triple $(x(\cdot),u(\cdot),y(\cdot))$ is called a generalized trajectory of \eqref{eq:SysNode} on  $\R_{\geq 0}$ if
\begin{itemize}
    \item $x(\cdot)\in C(\R_{\geq 0};\X)$, $u(\cdot)\in\L^2_{\mathrm{loc}}(\R_{\geq 0};\U)$, $y(\cdot)\in\L^2_{\mathrm{loc}}(\R_{\geq 0};\Y)$;
    \item there exists a sequence $(x_k(\cdot),u_k(\cdot),y_k(\cdot))$ of classical trajectories of \eqref{eq:SysNode} such that $x_k(\cdot)\to x(\cdot)$ in $C(\R_{\geq 0};\X)$, $u_k(\cdot)\to u(\cdot)$ in $\L^2_{\mathrm{loc}}(\R_{\geq 0};\U)$ and $y_k(\cdot)\to y(\cdot)$ in $\L^2_{\mathrm{loc}}(\R_{\geq 0};\Y)$.
\end{itemize}
By $u_k(\cdot) \to u(\cdot)$ in $\L^2_{\mathrm{loc}}(\R_{\ge 0};\U)$ we mean that $u_k|_K(\cdot) \to u|_K(\cdot)$ in $\L^2(K;\U)$ for every compact $K\subset\R_{\ge 0}$.
\end{definition}

Another possible additional property is \emph{well-posedness}. According to \cite[Def.~2.6]{Sta02}, a system node is well-posed if for all $t>0$, there exists a finite constant $c_t > 0$ such that all generalized trajectories $(x(\cdot),u(\cdot),y(\cdot))$ of $S$ satisfy
\begin{align*}
\Vert x(t)\Vert^2_\X + \Vert y(\cdot)\Vert_{\L^2([0,t];\Y)}^2 \leq c_t\left(\Vert x(0)\Vert_\X^2 + \Vert u(\cdot)\Vert_{\L^2([0,t];\U)}^2\right).
\end{align*}
If this constant $c_t$ can be chosen independent on $t$, then the system is callen {\em infinite-time well-posed}.

Note that well-posedness implies that for all $x_0\in\X$, $u(\cdot)\in\L^2_{\mathrm{loc}}(\R_{\geq 0};\U)$, there exist 
$x(\cdot)\in \L^2_{\mathrm{loc}}(\R_{\geq 0};\X)$, $y(\cdot)\in \L^2_{\mathrm{loc}}(\R_{\geq 0};\Y)$, such that $(x(\cdot),u(\cdot),y(\cdot))$ is called a generalized trajectory of \eqref{eq:SysNode} on  $\R_{\geq 0}$.

\section{Optimal control problem}\label{sec:OptControl}

Let $S:=\pmatsmall{A\& B\\ C\& D}$ be a system node on the spaces $(\U,\X,\Y)$ where $\U, \X$ and $\Y$ are Hilbert spaces. We consider infinite-dimensional systems of the form \eqref{eq:SysNode} and $x_0\in\X$,
we aim to minimize the values of the (possibly indefinite) cost functional $J(u(\cdot),y(\cdot))$ as in \eqref{eq:CostIntro}
 over $u(\cdot)\in\Lp[2](\mathbb{R}_{\geq 0};\U)$ such that $x(\cdot)$ and $y(\cdot)$ in \eqref{eq:SysNode} satisfy $x(0)=x_0$ and  $y(\cdot)\in\Lp[2](\R_{\geq 0};\Y)$, 
where $Q\in\Lb(\Y)$, $R\in\Lb(\U)$, $S\in\Lb(\U,\Y)$ with $Q=Q^*$ and $S=S^*$. We emphasize that no sign constraint is required on $s$, allowing the treatment of indefinite cost functionals. We start with the following assumption on \eqref{eq:SysNode}. For the notion of (bounded) quadratic form, we refer to \cite[Chap.~6]{kato1995perturbation}.
\begin{assumption}\label{assum:dissip}
Let $\mathcal{S}:\X\to\mathbb{R}$ be a bounded real quadratic form, and let $Q=Q^*\in\Lb(\Y)$, $S\in\Lb(\U,\Y)$, $R=R^*\in\Lb(\U)$. We suppose that \eqref{eq:SysNode} is dissipative with respect to the supply rate functional $s$ and storage function $\mathcal{S}$ in the sense of \cite[Def.~3.1]{Reis2025_Dissip}. In other words, for any $t\geq 0$ and any classical trajectory $(x(\cdot),u(\cdot),y(\cdot))$ of \eqref{eq:SysNode} on $[0,t]$, it holds that 
\begin{equation}
\mathcal{S}(x(t)) - \mathcal{S}(x(0)) + \int_0^t s(u(\tau),y(\tau))\d\tau\geq 0.
\label{eq:Dissip} 
\end{equation}
\end{assumption}
Note that in \cite{Reis2025_Dissip} unbounded storage functions are also considered. Here, however, we restrict ourselves to the simpler case of bounded storage functions. According to \cite[Thm.~4.1]{Reis2025_Dissip}, the inequality \eqref{eq:Dissip} yields the existence of a Hilbert space $\W$ and 
some $K\&L\in \Lb(\dom(A\&B),\W)$ with dense range, 
such that any classical trajectory $(x(\cdot),u(\cdot),y(\cdot))$ of \eqref{eq:SysNode} satisfies, for any $t\geq 0$,
\begin{equation}
\mathcal{S}(x(t)) - \mathcal{S}(x(0)) + \int_0^t s(u(\tau),y(\tau))\d\tau = \int_0^t\left\| K\&L\spvek{x(\tau)}{u(\tau)}\right\|_\W^2\d\tau.
\label{eq:DissipEq_DissipOutput}
\end{equation}
The function $w(\cdot):=K\&L\spvek{x(\cdot)}{u(\cdot)}:[0,t]\to \W$ is called {\em dissipation output}, and 
we may consider the extended system
\begin{equation}\label{eq:ODEnode_Extended}
  \spvek{\dot{x}(t)}{y(t)\\ w(t)}
  \;=\;
  \pmatsmall{A\& B\\[1mm] C\& D\\[1mm] K\& L}\spvek{x(t)}{u(t)}.
\end{equation}
Since $K\&L\in \Lb(\dom(A\&B),\W)$, the above system is also governed by a~system node.

Moreover, according to \cite[Thm.~4.1]{Reis2025_Dissip}, \eqref{eq:DissipEq_DissipOutput} is equivalent to the following operator-node Lur'e equation 
\begin{equation}
    2\re\langle Px_0,A\& B\left(\begin{smallmatrix}
        x_0\\ u_0
    \end{smallmatrix}\right)\rangle_\X + s(C\& D\left(\begin{smallmatrix}
        x_0\\ u_0
    \end{smallmatrix}\right),u_0) = \Vert K\& L\left(\begin{smallmatrix}
        x_0\\ u_0
    \end{smallmatrix}\right)\Vert_\W^2,
    \label{eq:RiccatiDissip}
\end{equation}
for all $\left(\begin{smallmatrix}
        x_0\\ u_0
    \end{smallmatrix}\right)\in\dom(A\& B)$. Hereby, $P\in\Lb(\X)$ is the operator that is associated to $\mathcal{S}$ in the sense that $\mathcal{S}(x_0) = \langle x_0,Px_0\rangle_\X$ for all $x_0\in\X$.
    
    \noindent Now we introduce the concept of an admissible state feedback for a system node, see~\cite[Def.~7.3.1]{Sta05book}. 
    \begin{definition}
    Let $S := \pmatsmall{A\& B\\C\& D}$ be a system node on the spaces $(\U,\X,\Y)$. An admissible state feedback for $S$ is an operator $F\& G\in\Lb(\dom(A\& B);\U)$ such that
    \begin{equation}
        \spvek{\dot{x}(t)\\y(t)}{z(t)} = \left[\begin{smallmatrix}A\& B\\[1mm]C\& D\\[1mm] F\& G\end{smallmatrix}\right]\spvek{x(t)}{u(t)}
        \label{eq:SysNodeExtraOut}
    \end{equation}
    in closed-loop with the feedback $u(\cdot) = z(\cdot) + v(\cdot)$ yields a system that is described by a system node with new input $v(\cdot)$ and output $\spvek{y(\cdot)}{z(\cdot)}$.
    \end{definition}
    Conditions guaranteeing admissibility of a state feedback can be found in \cite[Sec.~7.3]{Sta05book}. More precisely, admissibility of $F\& G$ is related to the invertibility of
    \begin{equation*}
        M := \pmatsmall{I_\X & 0\\[1mm]0 & I_\U} - \pmatsmall{0\\[1mm]F\& G}
        \label{eq:OpM}
    \end{equation*}
    viewed as a mapping from $\dom(A\& B)$ to $\dom(A_K\&B_K)$, where 
    \begin{equation}
    S_K = \pmatsmall{A_K\& B_K\\[1mm]C_K\& D_K\\[1mm] F_K\& G_K} := \pmatsmall{A\& B\\[1mm]C\& D\\[1mm] F\& G}M^{-1}
        \label{eq:SysNode_K}
    \end{equation}
    is a system node, see~\cite[Lem.~7.3.3]{Sta05book}. Moreover, the trajectories of 
    \begin{equation}
        \spvek{\dot{x}_K(t)\\y_K(t)}{z_K(t)} = S_K\spvek{x_K(t)}{v(t)}
        \label{eq:SysK}
    \end{equation}
    are related to those of \eqref{eq:SysNodeExtraOut} via the relation $u(\cdot) = z(\cdot) + v(\cdot)$, which in particular means that their state trajectories and their outputs coincide. The notion of admissible state feedback allows us to introduce the concept of \emph{state-feedback stabilizability} in the sense of~\cite[Def.~5.1]{Reis2025_Dissip}.
    \begin{definition}\label{def:StateFeedStab}
        Let $S := \pmatsmall{A\& B\\ C\& D}$ be a system node on the spaces $(\U,\X,\Y)$. $S$ is said to be \emph{state-feedback stabilizable} if there exists an admissible state feedback $F\& G\in\Lb(\dom(A\& B),\U)$ such that \eqref{eq:SysNode_K} with $u(\cdot) = z(\cdot) + v(\cdot)$ is infinite-time well-posed and the associated semigroup is strongly stable. 
    \end{definition}

    From now on, we make the following assumption on \eqref{eq:SysNode}.

    \begin{assumption}\label{assum:StateFeedbackStab}
        The system \eqref{eq:SysNode} is state-feedback stabilizable in the sense of Definition \ref{def:StateFeedStab}.
    \end{assumption}



The following results guarantee that, under Assumptions~\ref{assum:dissip} and \ref{assum:StateFeedbackStab}, the 
{\em value function}
    \begin{equation}
        \mathrm{Val}_J(x_0) = \inf\setdef*{J(u(\cdot),y(\cdot))}{\begin{aligned}& (x(\cdot),u(\cdot),y(\cdot))\in C(\R_{\ge0};\X)\times\L^2(\R_{\ge0};\U)\times\L^2(\R_{\ge0};\Y)\text{ is a }\\&
        \text{ generalized trajectory of \eqref{eq:SysNode} on $\R_{\geq 0}$ with $x(0) = x_0$ and $\lim_{t\to\infty}x(t)=0$}\end{aligned}}\label{eq:ValJ}
    \end{equation}
is a~bounded storage function. This in particular implies that $\mathrm{Val}_J(x_0)>-\infty$ for all $x_0\in \X$.
We will use the dissipation output $w(\cdot):=K\&L\spvek{x(\cdot)}{u(\cdot)}:[0,t]\to \W$. It follows immediately from the dissipation inequality that $w(\cdot)\in \L^2(\R_{\ge0};\W)$ whenever $(x(\cdot),u(\cdot),y(\cdot))$ is a~generalized trajectory of 
\eqref{eq:SysNode} on $\R_{\ge0}$ with bounded state trajectory $x(\cdot)$. We also consider the minimization of the $\L^2$-norm of the dissipation output. This leads to the consideration of the nonnegative cost functional
\begin{equation*}
    J_w(u(\cdot),w(\cdot)) := \int_0^\infty\Vert w(\tau)\Vert_\W^2\d\tau
    \label{eq:NewCost}
\end{equation*}
and associated value function
    \begin{equation}
        \mathrm{Val}_{J_w}(x_0) = \inf\setdef*{\|w(\cdot)\|_{\L^2}^2}{\begin{aligned}& \Big(x(\cdot),u(\cdot),\spvek{y(\cdot)}{w(\cdot)}\Big)\in C(\R_{\ge0};\X)\times\L^2(\R_{\ge0};\U)\times\L^2(\R_{\ge0};\Y\times\W) \text{ is }\\&
        \text{ a generalized trajectory of \eqref{eq:ODEnode_Extended} on $\R_{\geq 0}$ with $x(0) = x_0$ and $\lim_{t\to\infty}x(t)=0$}\end{aligned}}.\label{eq:ValJe}
    \end{equation}
Next we show a~relation between these two optimal control problems.
\begin{theorem}\label{thm:ReExprCost}
Consider a system of the form \eqref{eq:SysNode} where $S := \pmatsmall{A\& B\\ C\& D}$ is a system node on the spaces $(\U,\X,\Y)$ and suppose furthermore that $S$ satisfies Assumptions \ref{assum:dissip} and \ref{assum:StateFeedbackStab}. Further, let
$K\&L\in \Lb(\dom(A\&B),\W)$, $P=P^*\in\Lb(\X)$, such that \eqref{eq:RiccatiDissip} holds. 
Then, for all $x_0\in\X$, there exists $u(\cdot)\in\L^2(\R_{\geq 0};\U)$ such that, for some  $y(\cdot)\in\L^2(\R_{\geq 0};\Y)$ and $x(\cdot)\in C(\R_{\geq 0};\X)$ with $\lim_{t\to\infty}x(t)=0$, $(x(\cdot),u(\cdot),y(\cdot))$ is a~generalized trajectory of \eqref{eq:SysNode} in $\R_{\ge0}$. Further, for all such trajectories, the corresponding dissipation output is in $\L^2(\R_{\ge0};\W)$ with
\begin{equation}\label{eq:dissineqinfhor}
    \|w(\cdot)\|^2_{\L^2(\R_{\ge0};\W)} = -\mathcal{S}(x_0) + J(u(\cdot),y(\cdot))<\infty.
\end{equation}
\end{theorem}


\begin{proof}
Let $x_0\in\X$ be fixed. \\
{\em Step~1:}
We show that there exists $u(\cdot)\in\L^2(\R_{\geq 0};\U)$ such that, for some  $y(\cdot)\in\L^2(\R_{\geq 0};\Y)$ and $x(\cdot)\in C(\R_{\geq 0};\X)$ with $\lim_{t\to\infty}x(t)=0$, $(x(\cdot),u(\cdot),y(\cdot))$ is a~generalized trajectory of \eqref{eq:SysNode} in $\R_{\ge0}$. 
By Assumption~\ref{assum:StateFeedbackStab}, the system \eqref{eq:SysNode} is state-feedback stabilizable in the sense of Definition \ref{def:StateFeedStab}.
Hence, we find a~generalized trajectory of \eqref{eq:SysK} with zero input and initial state $x_0$, denoted by $\big(x_K(\cdot),0,\spvek{y_K(\cdot)}{z_K(\cdot)}\big)$, which has the properties 
$y_K(\cdot)\in \L^2(\R_{\geq 0};\Y)$, $z_K(\cdot)\in \L^2(\R_{\geq 0};\U)$, and
$x_K(\cdot)\in C(\R_{\geq 0};\X)$ with $\lim_{t\to\infty}x_K(t)=0$. Then we have that $(x(\cdot),u(\cdot),y(\cdot)):=(x_K(\cdot),z_K(\cdot),y_K(\cdot))$ is a~generalized trajectory of \eqref{eq:SysNode} in $\R_{\ge0}$ with the desired properties.

{\em Step~2:} We show the statement on the dissipation output. By definition of generalized trajectories, we have, for any $t>0$, that there exists a~sequence of classical trajectories $(x_n(\cdot),u_n(\cdot),y_n(\cdot))$ of \eqref{eq:SysNode} in $[0,t]$, such that $x_n(\cdot)$ converges pointwise to $x(\cdot)$, and $(u_n(\cdot))$, $(y_n(\cdot))$ converge to the restrictions of $u(\cdot)$ and $y(\cdot)$ to $[0,t]$, respectively. Denote the corresponding dissipation output sequence by $(w_n(\cdot))$.
Now the dissipation inequality gives
\[\forall\,n\in\N:\quad     \|w_n(\cdot)\|^2_{\L^2([0,t];\W)}= \mathcal{S}(x_0)-\mathcal{S}(x_n(t))+ \int_0^t s(u_n(\tau),y_n(\tau))\d\tau.
\]
Now taking the limit $n\to\infty$ gives that $w_t(\cdot):=K\&L\spvek{x(\cdot)}{u(\cdot)}\in \L^2([0,t];\W)$ is well-defined with 
\begin{equation}
\|w_t(\cdot)\|^2_{\L^2([0,t];\W)}= \mathcal{S}(x_0)-\mathcal{S}(x(t))+ \int_0^t s(u(\tau),y(\tau))\d\tau.\label{eq:dissinieqfinite}
\end{equation}
(Well-)defining $w(\cdot)\in \L^2_{\rm loc}(\R_{\ge0};\W)$ by $w\vert_{[0,t]}(\cdot)=w_t(\cdot)$ for all $t>0$, we obtain, by taking the limit $t\to\infty$ and using that $\lim_{t\to\infty}x(t)=0$, that  \eqref{eq:dissinieqfinite} implies that $w(\cdot)\in \L^2(\R_{\ge0};\W)$ with \eqref{eq:dissineqinfhor}.
\end{proof}

From now on, we will complement \eqref{eq:SysNode} with the dissipation output $w$, leading to the system
\eqref{eq:ODEnode_Extended}
with state $x(\cdot)$, input $u(\cdot)$ and output $\spvek{y(\cdot)}{w(\cdot)}$. 
In what follows, we will use the following notation
\begin{equation*}
    \tilde{s}\big(\spvek{y}{w},u\big)=:\left\langle \left(\begin{smallmatrix}
    \spvek{y}{w}\\ u
\end{smallmatrix}\right),\pmatsmall{0 & 0 & 0\\ 0 & I_\U & 0\\0 & 0 & 0} \left(\begin{smallmatrix}
    \spvek{y}{w}\\ u
\end{smallmatrix}\right)\right\rangle_{\Y\times \W\times \U}.
\label{eq:tilde_s}
\end{equation*}
It is clear that $\tilde{s}$ is a nonnegative quadratic form and that 
\[J_w(u(\cdot),w(\cdot)) = \int_0^\infty \tilde{s}((y(\tau),w(\tau)),u(\tau))\d\tau.\] 
It follows immediately that the value function in \eqref{eq:ValJe} can alternatively be rewritten as 
    \begin{equation}
        \mathrm{Val}_{J_w}(x_0) = \inf\setdef*{J_w(u(\cdot),w(\cdot))}{\begin{aligned}& \Big(x(\cdot),u(\cdot),\spvek{y(\cdot)}{w(\cdot)}\Big)\in C(\R_{\ge0};\X)\times\L^2(\R_{\ge0};\U)\times\L^2(\R_{\ge0};\Y\times\W)\text{ is }\\&
        \text{ a generalized trajectory of \eqref{eq:ODEnode_Extended} on $\R_{\geq 0}$ with $x(0) = x_0$ and $\lim_{t\to\infty}x(t)=0$}\end{aligned}}.\label{eq:ValJe_alt}
    \end{equation}
By invoking that $\tilde{s}(\cdot,\cdot)$ is nonnegative, \cite[Thm.~5.4]{Reis2025_Dissip} implies that
$\mathrm{Val}_{J_w}$ is a~bounded and quadratic storage function of \eqref{eq:ODEnode_Extended} with respect to the supply rate $\tilde{s}(\cdot,\cdot)$. Then, by \cite[Thm.~4.1]{Reis2025_Dissip}, there exists some self-adjoint and nonnegative $P_w\in\Lb(\X)$, such that 
\begin{equation*}\mathrm{Val}_{J_w}(x_0)=\scprod{x_0}{P_wx_0}_\X\quad\forall\,x_0\in\X,\end{equation*}
and there exists some Hilbert space $\W_w$ and some $E\&F\in\Lb(\dom A\&B;\W_w)$ with
\begin{equation}
        2\re\langle {P}_wx_0,A\& B\left(\begin{smallmatrix}
        x_0\\ u_0    \end{smallmatrix}\right)\rangle_\X + \tilde{s}\left(\sbvek{C\& D}{K\& L}\left(\begin{smallmatrix}
        x_0\\ u_0    \end{smallmatrix}\right),u_0\right) = \Vert E\& F\left(\begin{smallmatrix}
        x_0\\ u_0   \end{smallmatrix}\right)\Vert_{\W_w}^2\quad\forall\, \spvek{x_0}{u_0}\in\dom(A\& B).
        \label{eq:Jwdiss}
        \end{equation}
Theorem~\ref{thm:ReExprCost} can be used to show the following the link between $\mathrm{Val}_J$ and $\mathrm{Val}_{J_w}$.



\begin{proposition}\label{prop:ValueFct}
Consider a system of the form \eqref{eq:SysNode} where $S := \pmatsmall{A\& B\\ C\& D}$ is a system node on the spaces $(\U,\X,\Y)$ and suppose furthermore that $S$ satisfies Assumptions \ref{assum:dissip} and \ref{assum:StateFeedbackStab}. Further, let
$K\&L\in \Lb(\dom(A\&B),\W)$, $P=P^*\in\Lb(\X)$, such that \eqref{eq:RiccatiDissip} holds.
Then the value functions \eqref{eq:ValJ} and \eqref{eq:ValJe_alt} are related by 
 \begin{equation}
\mathrm{Val}_J(x_0) =  \mathrm{Val}_{J_w}(x_0)+\mathcal{S}(x_0)\quad\forall x_0\in\X.\label{eq:ValJident}
 \end{equation}
In particular, $\mathrm{Val}_J$ is a~bounded quadratic form on $\X$.
\end{proposition}

\begin{proof}
As discussed before this proposition, $\mathrm{Val}_{J_w}$ is a bounded quadratic functional. The identity \eqref{eq:ValJident} follows from Theorem~\ref{thm:ReExprCost} by an infimization of \eqref{eq:dissineqinfhor} over all trajectories of \eqref{eq:ODEnode_Extended} with $x(0)=x_0$ and $u\in \L^2(\R_{\ge0};\U)$, $y\in \L^2(\R_{\ge0};\Y)$. In particular, 
$\mathrm{Val}_J$ is quadratic, since it is the sum of two quadratic forms.
    \end{proof}
Next, we use that $\mathrm{Val}_{J_w}$ is a bounded quadratic storage function of the extended system \eqref{eq:ODEnode_Extended} with supply rate $\tilde{s}(\cdot,\cdot)$ to show that $\mathrm{Val}_{J}$ is likewise a storage function for \eqref{eq:SysNode}, now equipped with the (possibly indefinite) supply rate $s(\cdot,\cdot)$. Moreover, we show that this storage function is maximal in a suitable sense.
        \begin{theorem}\label{thm:MainThm}
Consider a system of the form \eqref{eq:SysNode} where $S := \pmatsmall{A\& B\\ C\& D}$ is a system node on the spaces $(\U,\X,\Y)$ and suppose furthermore that $S$ satisfies Assumptions \ref{assum:dissip} and \ref{assum:StateFeedbackStab}. Then the value function  $\mathrm{Val}_J$ in \eqref{eq:ValJ} is a~bounded and quadratic storage function of \eqref{eq:SysNode} with respect to the supply rate $s(\cdot,\cdot)$. Moreover, any further bounded quadratic storage function $\mathrm{V}$ of \eqref{eq:SysNode} with respect to the supply rate $s(\cdot,\cdot)$ fulfills
\[\mathrm{Val}_J(x_0)\geq\mathrm{V}(x_0)\quad\forall\,x_0\in\X.\]


    \end{theorem}

    \begin{proof}
        Let $P$ and $K\& L$ be as in \eqref{eq:RiccatiDissip}. 
        Then, as elaborated prior to Proposition~\ref{prop:ValueFct}, there exists some self-adjoint and nonnegative $P_w\in\Lb(\X)$,  some Hilbert space $\W_w$ and some $E\&F\in\Lb(\dom (A\&B);\W_w)$, such that \eqref{eq:Jwdiss} holds. Expanding $\tilde{s}(\cdot,\cdot)$
        in \eqref{eq:Jwdiss} then gives rise to
\begin{equation}
        2\re\langle {P}_wx_0,A\& B\left(\begin{smallmatrix}
        x_0\\ u_0    \end{smallmatrix}\right)\rangle_\X + \Vert K\& L\spvek{x_0}{u_0}\Vert_\W^2 = \Vert E\& F\left(\begin{smallmatrix}
        x_0\\ u_0   \end{smallmatrix}\right)\Vert_{{\W}_w}^2\quad\forall\spvek{x_0}{u_0}\in\dom (A\&B).\label{eq:dissipineq0}
\end{equation}
Now an addition of \eqref{eq:RiccatiDissip}
gives 
\begin{equation}
        2\re\langle (P+{P}_w)x_0,A\& B\left(\begin{smallmatrix}
        x_0\\ u_0    \end{smallmatrix}\right)\rangle_\X + s(C\& D\left(\begin{smallmatrix}
        x_0\\ u_0
    \end{smallmatrix}\right),u_0) = \Vert E\& F\left(\begin{smallmatrix}
        x_0\\ u_0   \end{smallmatrix}\right)\Vert_{{\W}_w}^2\quad\forall\spvek{x_0}{u_0}\in\dom (A\&B).\label{eq:dissineqtotal}
\end{equation}
Since $\mathcal{S}(x_0)=\scprod{x_0}{Px_0}$ and $\mathrm{Val}_{J_w}(x_0)=\scprod{x_0}{P_wx_0}$ for all $x_0\in\X$, Proposition~\ref{prop:ValueFct}
        yields that 
\[\mathrm{Val}_{J}(x_0)=\scprod{x_0}{(P+P_w)x_0}_\X\quad \forall\,x_0\in\X.\]
Combining the latter with \eqref{eq:dissineqtotal}, \cite[Thm.~4.1]{Reis2025_Dissip} implies that $\mathrm{Val}_{J}$ is a~storage function of \eqref{eq:SysNode} with respect to the supply rate $s(\cdot,\cdot)$.\\
To prove the maximality property of $\mathrm{Val}_{J}$, assume that $\mathrm{V}$ is another bounded quadratic storage function of \eqref{eq:SysNode} with respect to the supply rate $s(\cdot,\cdot)$. Then there exists some $\tilde{P}=\tilde{P}^*\in\Lb(\X)$, such that 
\[\mathrm{V}(x_0)=\scprod{x_0}{\tilde{P}x_0}_\X\quad \forall\,x_0\in\X.\]
and \cite[Thm.~4.1]{Reis2025_Dissip} implies that
\[        2\re\langle \tilde{P}x_0,A\& B\left(\begin{smallmatrix}
        x_0\\ u_0    \end{smallmatrix}\right)\rangle_\X + s(C\& D\left(\begin{smallmatrix}
        x_0\\ u_0
    \end{smallmatrix}\right),u_0) \ge0\quad\forall\spvek{x_0}{u_0}\in\dom (A\&B).\]
Now subtracting \eqref{eq:RiccatiDissip} gives
\[        2\re\langle (\tilde{P}-P)x_0,A\& B\left(\begin{smallmatrix}
        x_0\\ u_0    \end{smallmatrix}\right)\rangle_\X + \tilde{s}(\sbvek{C\& D}{K\&L}\left(\begin{smallmatrix}
        x_0\\ u_0
    \end{smallmatrix}\right),u_0) \ge0\quad\forall\spvek{x_0}{u_0}\in\dom (A\&B).\]
This means that $x_0\mapsto \scprod{x_0}{(\tilde{P}-P)x_0}_\X$ is a~storage function to the extended system \eqref{eq:ODEnode_Extended} with respect to the nonnegative supply rate $\tilde{s}(\cdot,\cdot)$. Then \cite[Thm.~5.4]{Reis2025_Dissip} implies that 
\[\scprod{x_0}{(\tilde{P}-P)x_0}_\X\leq \scprod{x_0}{P_wx_0}_\X.\]
Therefore, for all $x_0\in\X$, we have
\[\mathrm{V}(x_0)=\scprod{x_0}{\tilde{P}x_0}_\X\leq \scprod{x_0}{(P+P_w)x_0}_\X.\]
 Invoking that, by Proposition~\ref{prop:ValueFct}, the latter expression equals to $\mathrm{Val}_{J}$, the statement is proven. 
    \end{proof}
   \begin{remark}
       Impedance and scattering passive systems are dissipative with respect to the supply rate functionals $s(y,u) = 2\re\langle u,y\rangle$ and $s(y,u) = \Vert u\Vert^2-\Vert y\Vert^2$, respectively, and storage function $\mathcal{S}(x_0) = -\Vert x_0\Vert^2_\X$, whence $P = -I_\X$. This means that minimizing the cost $J(u(\cdot),y(\cdot)) = \int_0^\infty s(y(\tau),u(\tau))\d\tau$ in such settings is equivalent in minimizing $J_w(u(\cdot),w(\cdot)) = \int_0^\infty \Vert w(\tau)\Vert_\W^2\d\tau$, where $w$ is the dissipation output, also expressed as $w = K\& L\spvek{x}{u}$, for some $K\&L\in\mathcal{L}_b(\dom(A\& B),\W)$ and some Hilbert space $\W$, provided that Assumption \ref{assum:StateFeedbackStab} is satisfied. Furthermore, the value functions associated to $J$ and $J_w$, denoted by $\mathrm{Val}_J(x_0)$ and $\mathrm{Val}_{J_w}(x_0)$, respectively, satisfy $\mathrm{Val}_J(x_0) + \Vert x_0\Vert_\X = \mathrm{Val}_{J_w}(x_0) = \langle x_0,P_wx_0\rangle_\X$ for some self-adjoint operator $P_w\in\mathcal{L}_b(\X)$, see Prop.~\ref{prop:ValueFct}. Thm.~\ref{thm:MainThm} then yields the existence of a Hilbert space $\W_w$ and an operator $E\& F\in\mathcal{L}_b(\dom(A\& B),\W_w)$ such that $P_w$ satisfies the operator-node Lur'e equation \eqref{eq:dissineqtotal}. According to $P = -I_\X$, \eqref{eq:dissineqtotal} reduces to 
       \begin{equation*}2\re\langle P_wx_0,A\& B\spvek{x_0}{u_0}\rangle_\X + s(C\& D\spvek{x_0}{u_0},u_0) = \Vert E\& F\spvek{x_0}{u_0}\Vert_{W_w}^2 + 2\re \langle x_0,A\& B\spvek{x_0}{u_0}\rangle_\X
       \end{equation*}
       for all $\spvek{x_0}{u_0}\in\dom(A\& B)$.       
   \end{remark}
\begin{remark}
In this article, our primary focus is on the value function. Accordingly, we have not discussed the minimizing input (i.e., the optimal control) in detail. First, it should be noted that an optimal control need not exist; in general, one may only have an infimizing (not necessarily convergent) sequence.
Nevertheless, Theorem~\ref{thm:MainThm}, in particular \eqref{eq:dissineqinfhor}, implies that, if an optimal control $u^{\mathrm{opt}}(\cdot) \in \L^2(\R_{\ge 0};\,\cdot)$ exists, then it minimizes the cost functional $J(\cdot,\cdot)$ for the system with initial condition $x(0)=x_0$ if and only if the associated dissipation output $w^{\mathrm{opt}}(\cdot)$ has minimal $\L^2$-norm.
Using \cite[Rem.~5.6]{Reis2025_Dissip}, which treats optimal control problems with semidefinite cost functionals, it follows that an optimal control is characterized by the property that, for the operator $E\&F$ in \eqref{eq:dissipineq0}, the corresponding state trajectory $x^{\mathrm{opt}}(\cdot)$ satisfies
\[
E \& F \, \spvek{x^{\mathrm{opt}}(\cdot)}{u^{\mathrm{opt}}(\cdot)} = 0.
\]

Altogether, this shows that an optimal control for \eqref{eq:SysNode} with cost functional \eqref{eq:CostIntro} and initial value $x_0$ is characterized by the infinite-dimensional differential-algebraic system
\[
\spvek{\dot{x}^{\mathrm{opt}}(t)}{0}
=
\pmatsmall{A\& B\\[1mm] E\& F}
\spvek{x^{\mathrm{opt}}(t)}{u^{\mathrm{opt}}(t)},\quad x^{\mathrm{opt}}(0)=x_0
\]
\end{remark}

\section{Examples}\label{sec:Examples}

\subsection{A transport equation with boundary input and output}
    We consider the partial differential equation 
    \begin{equation}
    \label{eq:Shift}
    \begin{aligned}
        \tfrac{\partial x}{\partial t}(\xi,t) &= -\tfrac{\partial x}{\partial\xi}(\xi,t),\,\, x(\xi,0) = x_0(\xi),\\
        u(t) &= x(0,t),\qquad y(t) = x(1,t),
    \end{aligned}
    \end{equation}
    $t\geq 0$, $\xi\in[0,1]$, where $x_0\in\X := \L^2([0,1];\C)$, $u(t)\in\C$ and $y(t)\in\C$ are the input and the output, respectively. For a fixed $x_0 \in\X$, we want to minimize $J(u(\cdot),y(\cdot))$, where
    \begin{equation*}
        J(u(\cdot),y(\cdot)) = \int_0^\infty 2\vert y(\tau)\vert^2 - \vert u(\tau)\vert^2\d\tau,
        \label{eq:cost_Example_Shift}
    \end{equation*}
    among those $u(\cdot)\in\L^2(\R_{\geq 0};\C)$ for which $y(\cdot)\in\L^2(\R_{\geq 0};\C)$ in \eqref{eq:Shift}. Observe that \eqref{eq:Shift} can be written within the system node formalism with 
    \begin{align*}
        &A\&B : \dom(A\& B)\to\X,\quad A\& B\spvek{x_0}{u_0} = -\tfrac{\dx_0}{\d\xi}\\
        &C\& D:\dom(A\& B)\to\C,\quad C\&D\spvek{x_0}{u_0} = x_0(1),
    \end{align*}
    where $\dom(A\& B) = \left\{\spvek{x_0}{u_0}\in\H^1([0,1];\C)\times\C, x_0(0) = u_0\right\}$.
    Now consider $\mathcal{S}(x_0) = \Vert x_0\Vert^2_\X$. Developing the left-hand side of \eqref{eq:RiccatiDissip} with $P = I_\X$ and $\spvek{x_0}{u_0}\in\dom(A\& B)$ yields
    \begin{align*}
        &2\re\langle x_0,A\& B\spvek{x_0}{u_0}\rangle_\X + s(C\& D\spvek{x_0}{u_0},u_0)
        = \vert x_0(0)\vert^2 - \vert x_0(1)\vert^2 + 2\vert C\& D\spvek{x_0}{u_0}\vert^2 - \vert u_0\vert^2 = \vert C\&D\spvek{x_0}{u_0}\vert^2,
    \end{align*}
    which means that \eqref{eq:Shift} is dissipative with respect to the supply rate functional $s(y,u) = 2\vert y\vert^2-\vert u\vert^2$ and storage function $\mathcal{S}(x_0) = \Vert x_0\Vert^2_\X$. The corresponding operator $K\& L$ is given by $K\& L:\dom(A\& B)\to\C, K\& L\spvek{x_0}{u_0} = C\& D\spvek{x_0}{u_0}$. Moreover, it can be shown that \eqref{eq:Shift} is scattering passive, hence infinite-time well-posed. In addition, \eqref{eq:Shift} with $u(\cdot)\equiv 0$ yields a system with the property that $x(t)\to 0$ when $t\to\infty$. In particular, this implies that Assumption \ref{assum:StateFeedbackStab} is satisfied. Consequently, our minimization problem is equivalent to the minimization of $J_w(u(\cdot),w(\cdot)) = \int_0^\infty\vert w(\tau)\vert\d\tau = \int_0^\infty \vert y(\tau)\vert^2\d\tau$, subject to \eqref{eq:Shift}. According to \cite{HastirJacobZwartSIAM25}, the unique optimal control to this optimization problem is given by $u^{\opt}(\cdot)\equiv 0$ and the optimal cost is $\mathrm{Val}_{J_w}(x_0) = \Vert x_0\Vert^2_\X$, which suggests that $P_w$ in \eqref{eq:dissineqtotal} is $P_w = I_\X$. Now let us analyze the left-hand side of \eqref{eq:dissineqtotal} for $\spvek{x_0}{u_0}\in\dom(A\& B)$. It holds that
    \begin{multline*}
        2\re\langle (P + P_w)x_0,A\& B\spvek{x_0}{u_0}\rangle_\X + s(C\& D\spvek{x_0}{u_0},u_0)\\
        = 2\vert x_0(0)\vert^2 - 2\vert x_0(1)\vert^2 + 2\vert C\& D\spvek{x_0}{u_0}\vert^2 - \vert u_0\vert^2 = \vert u_0\vert^2.
    \end{multline*}
    This means that the statement of Theorem \ref{thm:MainThm} holds with $E\& F:\dom(A\& B)\to\C, E\& F\spvek{x_0}{u_0} = u_0$. The closed-loop system obtained by interconnecting \eqref{eq:Shift} with the optimal input $u^{\opt}(\cdot) \equiv 0$ is given by 
    \begin{subequations}
    \begin{align*}
        \tfrac{\partial x^{\opt}}{\partial t}(\xi,t) &= -\tfrac{\partial x^{\opt}}{\partial\xi}(\xi,t),\qquad x^{\opt}(\xi,0) = x_0(\xi),\\
        x^{\opt}(0,t) &= 0,\,\,x^{\opt}(1,t) = y(t).
    \end{align*}
    \end{subequations}
    It is easy to deduce that $x^{\opt}(\xi,t) = x_0(1-\xi+t)\mathbf{1}_{[0,\xi]}(t)$, which implies that 
    \begin{align*}
        \mathrm{Val}_{J_w}(x_0) &= \int_0^\infty\vert y(\tau)\vert^2\d\tau = \Vert x_0\Vert^2_\X = \langle x_0,P_wx_0\rangle_\X\\
        &=2\Vert x_0\Vert^2_\X - \mathcal{S}(x_0)\\
        &= \mathrm{Val}_J(x_0) - \mathcal{S}(x_0),
    \end{align*}
    which confirms the statement of Theorem \ref{thm:MainThm}.

\subsection{A wave equation with force control and velocity measurement}
As second example, we consider the following vibrating string
\begin{equation}
    \tfrac{\partial^2 \bm w}{\partial t^2}(\xi,t) = \tfrac{\partial^2 \bm w}{\partial \xi^2}(\xi,t),\qquad t\geq 0, \xi\in [0,1]
    \label{eq:String}
\end{equation}
where $\bm w(\xi,t)$ is the vertical displacement of the string at time $t$ and position $\xi$. We assume that the string is damped at its left hand-side; this is modelled by the boundary condition
\begin{equation}
    \tfrac{\partial \bm w}{\partial t}(0,t) = \tfrac{\partial \bm w}{\partial\xi}(0,t).
    \label{eq:StringDampedLeft}
\end{equation}
Moreover, the force is controlled at $\xi = 1$ and the velocity is measured at the same boundary point, i.e.
\begin{equation}
        u(t) = \tfrac{\partial \bm w}{\partial\xi}(1,t),\quad
        y(t) = \tfrac{\partial \bm w}{\partial t}(1,t).\label{eq:String_InOut}
\end{equation}
Our aim for this example is to minimize the supplied power, that is, to minimize the quantity
\begin{equation*}
    J(u(\cdot),y(\cdot)) = \int_0^\infty 2\re\langle u(\tau),y(\tau)\rangle_\C\d\tau.
\end{equation*}
With the variables $x_1(\xi,t) := \frac{\partial \bm w}{\partial t}(\xi,t)$ and $x_2(\xi,t) := \frac{\partial \bm w}{\partial\xi}(\xi,t)$, \eqref{eq:String} with \eqref{eq:StringDampedLeft} and \eqref{eq:String_InOut} is written equivalently as
\begin{equation}
\label{eq:String_ChgVar}
    \begin{aligned}
        \tfrac{\partial}{\partial t}\spvek{x_1}{x_2}(\xi,t) &= \pmatsmall{0 & 1\\[1mm]1 & 0}\tfrac{\partial}{\partial \xi}\spvek{x_1}{x_2}(\xi,t),\quad
        0 = x_1(0,t)-x_2(0,t),\\
        u(t) &= x_2(1,t),\quad
        y(t)  = x_1(1,t).
    \end{aligned}
\end{equation}
As state space, we consider $\X := \L^2([0,1];\C^2)$ equipped with the usual inner product. The input and output spaces are $\U = \C$ and $\Y = \C$, respectively. The above system admits a system node representation of the form \eqref{eq:SysNode} with 
\begin{align*}
    A\& B\spvek{x_0}{u_0} &= \pmatsmall{0 & 1\\[1mm]1 & 0}\tfrac{\d x_0}{\d\xi},\\
    \dom(A\& B) &= \left\{\spvek{\spvek{x_{0,1}}{x_{0,2}}}{u_0}\in\H^1([0,1];\C^2)\times\U, x_{0,1}(0) = x_{0,2}(0), u_0 = x_{0,2}(1)\right\},\\
    C\&D\spvek{\spvek{x_{0,1}}{x_{0,2}}}{u_0} &= x_{0,1}(1).
\end{align*}
Considering $\spvek{x_0}{u_0}\in\dom(A\& B)$, we have, by again denoting $\spvek{x_{0,1}}{x_{0,2}}$, that
\begin{align*}
    \re\langle A\& B\spvek{x_0}{u_0},x_0\rangle_\X &=\re\int_0^1 x_0(\xi)^*\pmatsmall{0 & 1\\[1mm]1 & 0}\tfrac{\d x_0}{\d\xi}(\xi)\d\xi\\
    &= \re\int_0^1 \overline{x_{0,1}(\xi)}\tfrac{\d x_{0,2}}{\d \xi}(\xi)\d\xi + \re\int_0^1 \overline{x_{0,2}(\xi)}\tfrac{\d x_{0,1}}{\d\xi}(\xi)\d\xi\\
    &= \re\langle C\& D\spvek{x_0}{u_0},u_0\rangle_\U - \vert x_{0,1}(0)\vert^2,
\end{align*}
which, according to \cite[Thm.~4.2]{Sta02}, means that \eqref{eq:String_ChgVar} is impedance passive. Moreover, the previous calculation implies that
\begin{equation}
-2\re\langle A\& B\spvek{x_0}{u_0},x_0\rangle_\X + 2\re\langle C\& D\spvek{x_0}{u_0},u_0\rangle_\U = 2\vert x_{0,1}(0)\vert^2,
\label{eq:DissipWave}
\end{equation}
for all $\spvek{x_0}{u_0}\in\dom(A\& B)$. Hence, \eqref{eq:String_ChgVar} is dissipative with respect to the supply rate functional $s(y,u) = 2\re(\overline{ u}{y})$ (which corresponds to $S=1$ and $Q=R=0$) and bounded storage function $\mathcal{S}(x_0) = -\Vert x_0\Vert_\X^2$. The relation \eqref{eq:DissipWave} suggests that $K\& L\spvek{x_0}{u_0} = \sqrt{2}x_{0,1}(0)$. In order to convert our optimization problem, one needs to show that Assumption \ref{assum:StateFeedbackStab} holds. For this, impedance passivity of $\pmatsmall{A\& B\\ C\& D}$ implies that the output feedback $u(\cdot) = -C\&D\spvek{x(\cdot)}{u(\cdot)} + v(\cdot)$, where $v(\cdot)$ is an external input yields a system that is infinite-time well posed and whose semigroup generator is given by 
\begin{subequations}
\begin{align*}
    A^K x_0 &= \pmatsmall{0 & 1\\[1mm]1 & 0}\tfrac{\d x_0}{\d\xi}\\
    \dom(A^K) &= \left\{x_0\in\H^1([0,1];\C^2), x_{0,1}(0) = x_{0,2}(0), x_{0,1}(1) = -x_{0,2}(1)\right\},
\end{align*}
\end{subequations}
see \cite[Thm.~2.5]{HasPau2025}. Now we take $x_0\in\dom(A^K)$. It holds that
\begin{align*}
    \re\langle A^Kx_0,x_0\rangle_\X = \re\int_0^1 x_0(\xi)^*\pmatsmall{0 & 1\\[1mm]1 & 0}\tfrac{\d x_0}{\d\xi}(\xi)\d\xi 
    &= \re\big(\overline{x_{0,1}(1)}x_{0,2}(1)\big)
    - \re\big(\overline{x_{0,1}(0)}x_{0,2}(0)\big)\\
    &= -|x_{0,1}(1)|^2-\frac{1}{2}|x_{0,1}(0)|^2-\frac{1}{2}|x_{0,2}(0)|^2.
\end{align*}
As a consequence, \cite[Lem.~9.1.4]{JacobZwart} implies that $A^K$ is the generator of an exponentially stable semigroup, whence Assumption \ref{assum:StateFeedbackStab} is satisfied. Hence, minimizing $J$ is equivalent to minimizing $J_w$ given by 
\begin{equation}
    J_w(u(\cdot),w(\cdot)) = \int_0^\infty \vert w(\tau)\vert^2\d\tau = 2\int_0^\infty \vert x_{1}(0,\tau)\vert^2\d\tau.
    \label{eq:String_NewCost}
\end{equation}
According to \cite{HastirJacobZwartSIAM25}, it can be shown that the unique optimal control that minimizes \eqref{eq:String_NewCost} subject to \eqref{eq:String_ChgVar} is given by\footnote{This optimal control can be obtained by applying the results of \cite{HastirJacobZwartSIAM25} to the transformed system
$\frac{\partial X}{\partial t}(\xi,t) = -\frac{\partial X}{\partial\xi}(\xi,t)$, with
\begin{align*}
    \spvek{0}{u(t)} &= \pmatsmall{0 & -2\\[1mm]1 & 0}X(0,t) + \pmatsmall{0 & 0\\[1mm]0 & 1}X(1,t)\\
    w(t) &= \spvek{0 & -\sqrt{2}}{}X(0,t) + \spvek{\sqrt{2} & 0}{}X(1,t),
    \end{align*}
    where $X_1(\xi,t) = \frac{1}{2}\left[x_1(1-\xi,t) + x_2(1-\xi,t)\right]$ and $X_2(\xi,t) = \frac{1}{2}\left[-x_1(\xi,t)+x_2(\xi,t)\right]$. The optimal input for this transformed system is given by $u^{\opt}(t) = X_2^{\opt}(1,t)$, which corresponds to $u^{\opt}(t) = \frac{1}{2}\left[x_2^{\opt}(1,t) - x_1^{\opt}(1,t)\right]$.
}
\begin{equation}
    u^{\opt}(t) = \tfrac{1}{2}\left(x_2^{\opt}(1,t) - x_1^{\opt}(1,t)\right),
    \label{eq:StringOptInp}
\end{equation}
where $x^{\opt}(\cdot)$ is subject to\footnote{This optimal system is obtained by replacing $u(\cdot)$ by $u^{\opt}(\cdot)$ in \eqref{eq:String_ChgVar}.}
\begin{equation}
\label{eq:String_OptSys}
\begin{aligned}
        \tfrac{\partial}{\partial t}\spvek{x_1^{\opt}}{x_2^{\opt}}(\xi,t) &= \pmatsmall{0 & 1\\[1mm]1 & 0}\tfrac{\partial}{\partial \xi}\spvek{x_1^{\opt}}{x_2^{\opt}}(\xi,t),\\
        x_1^{\opt}(0,t) &= \phantom{-}x_2^{\opt}(0,t)\\
        x_1^{\opt}(1,t) &= -x_2^{\opt}(1,t).
    \end{aligned}
    \end{equation}
The corresponding optimal cost operator, denoted by $P_w$, is given by 
\begin{equation*}
    P_w = \tfrac{1}{2}\pmatsmall{I & I\\[1mm]I & I}.
    \label{eq:StringOptCostOp}
\end{equation*}
Let us start with the verification of the operator-node Lur'e equation \eqref{eq:dissineqtotal}. By considering $\spvek{x_0}{u_0}\in\dom(A\& B)$ and denoting $x_0:=\spvek{x_{0,1}}{x_{0,2}}$, it holds that
\begin{align*}
    &2\re\scprod{(P + P_w)x}{A\& B\spvek{x_0}{u_0}}_\X + s(C\&D\spvek{x_0}{u_0},u_0)\\
    &= \re\scprod*{\pmatsmall{-I & \phantom{-}I\\ \phantom{-}I & -I}x_0}{A\& B\spvek{x_0}{u_0}}_\X + 2\re(\,\overline{x_{0,1}(1)}\,u_0\,)\\
    &= \int_0^1 
    \spvek{~\\[1mm]x_{0,2}-x_{0,1}\\[0mm]}{x_{0,1}-x_{0,2}\\[1mm]}^*
    \spvek{\frac{\d x_{0,2}}{\d\xi}}{ \frac{\d x_{0,1}}{\d\xi}}\d\xi + 2\re(\,\overline{x_{0,1}(1)}\,u_0\,)\\
    &= \tfrac{1}{2}|x_{0,2}(1)|^2 - \tfrac{1}{2}|x_{0,2}(0)|^2 + \tfrac{1}{2}|x_{0,1}(1)|^2\\&\qquad-\tfrac{1}{2}|x_{0,1}(0)|^2 - \re(\,\overline{x_{0,1}(1)}x_{0,2}(1)\,) + \re(\,\overline{x_{0,1}(0)}x_{0,2}(0)\,) + 2\re(\,\overline{x_{0,1}(1)} u_0)\\
    &= \tfrac{1}{2}|x_{0,2}(1)|^2 + \tfrac{1}{2}|x_{0,1}(1)|^2 + \re(\,\overline{x_{0,1}(1)}u_0\,)\\
    &= \tfrac{1}{2}|u_0 + x_{0,1}(1)|^2=: \Vert E\& F\spvek{x_0}{u_0}\Vert_{\W_w}^2,
\end{align*}
where $\W_w = \C$ and $E\& F\spvek{x_0}{u_0}=\tfrac1{\sqrt{2}}\big(u_0 + x_{0,1}(1)\big)$. According to \cite{Reis2025_Dissip}, the optimal input is obtained by imposing the additional constraint 
\[E\& F\spvek{x^{\opt}_0}{u^{\opt}_0}=0.\] This leads to the feedback law $u^{\opt}_0 = -x_{0,1}^{\opt}(1)$. According to \eqref{eq:StringOptInp} and to \eqref{eq:String_OptSys}, it holds that \[u^{\opt}_0 = \frac{1}{2}\left(x_{0,2}^{\opt}(1) - x_{0,1}^{\opt}(1)\right) = \frac{1}{2}\left(-x_{0,1}^{\opt}(1) - x_{0,1}^{\opt}(1)\right) = -x_{0,1}^{\opt}(1),\] which confirms what is obtained by setting $E\&F\spvek{x^{\opt}_0}{u^{\opt}_0}$ to $0$. 
Let us now verify the link between the value functions. We have that $\mathrm{Val}_{J_w}(x_0) = \langle x_0,P_wx_0\rangle_\X$. Moreover, it holds that
\begin{equation*}
    J(u^{\opt}(\cdot),y^{\opt}(\cdot)) = -2\int_0^\infty \vert x_1^\opt(1,\tau)\vert^2\d\tau = \mathrm{Val}_J(x_0).
\end{equation*}
By looking at \eqref{eq:String_OptSys}, it can be shown that $x_1^{\opt}(1,t) = -\frac{1}{2}(x_{0,2}(t) - x_{0,1}(t))\mathbf{1}_{[0,1]}(t)$, where $x_0 := \spvek{x_{0,1}}{x_{0,2}}$ is the initial condition supplied to \eqref{eq:String_OptSys}. Consequently, we have that
\begin{align*}
    \mathrm{Val}_{J_w}(x_0) + \mathcal{S}(x_0) &= \frac{1}{2}\int_0^1 \vert x_{0,1}(\xi) + x_{0,2}(\xi))\vert^2\d \xi - \int_0^1 \vert x_{0,1}(\xi)\vert^2\d \xi - \int_0^1 \vert x_{0,2}(\xi)\vert^2\d \xi\\
    &= \int_0^1 -\frac{1}{2}\vert x_{0,1}(\xi)\vert^2-\frac{1}{2}\vert x_{0,2}(\xi)\vert^2 + \re(\,\overline{x_{0,1}(\xi)}x_{0,2}(\xi)\,)\d\xi\\
    &= -\frac{1}{2}\int_0^1\vert x_{0,1}(\xi)-x_{0,2}(\xi)\vert^2\d\xi\\
    &= -2\int_0^\infty \vert x_1^{\opt}(1,\tau)\vert^2\d\tau\\
    &= \mathrm{Val}_J(x_0),
\end{align*}
which confirms the statement of Proposition \ref{prop:ValueFct}.

\subsection{A boundary controlled heat equation}
    In this last example, we consider a heat equation on a bounded Lipschitz domain $\Omega\subset\R^d$, $d\ge1$.
    We consider the heat equation with Dirichlet boundary control and Dirichlet output, i.e.,
    \begin{equation*}
        \label{eq:Heat}
        \begin{aligned}
            \frac{\partial x}{\partial t}(\xi,t) &= \Delta x(\xi,t),& \xi\in\Omega,\\
            \frac{\partial x}{\partial n}(\xi,t) &= y(\xi,t),&\xi\in\partial\Omega,\\            x(\xi,t)&=u(\xi,t),&\xi\in\partial\Omega,
        \end{aligned}
    \end{equation*}
    where $\frac{\partial x}{\partial n}(\xi,t) = n(\xi)\cdot\nabla x(\xi,t)$ with $n$ being the outward normal to $\partial\Omega$. 
    It is straightforward to show that this constitutes a~system node with $\X=\L^2(\Omega)$, $\Y=\H^{1/2}(\partial\Omega)$, and $\U=\H^{-1/2}(\partial\Omega)$, where the latter is the {\em anti-dual} of $\H^{1/2}(\partial\Omega)$, i.e., the set of bounded conjugate-linear functionals on $\H^{1/2}(\partial\Omega)$. The system node is given by
    \begin{equation*}
    \label{eq:HeatSysNode}
    \begin{aligned}
       S\spvek{x_0}{u_0} &= \pmatsmall{A\& B\\[1mm]C\& D}\pvek{x_0}{u_0} = \spvek{\Delta x_0}{\operatorname{tr} x_0}\\
       \dom(S) &= \left\{\spvek{x_0}{u_0}\in\H^1(\Omega)\times \H^{-1/2}(\partial\Omega)\,\vert\,\nabla x_0\in\H_{\operatorname{div}}(\Omega;\C^d)\,\wedge\,\operatorname{tr}_n\nabla x_0 = u\right\},
    \end{aligned}
    \end{equation*}
where $\H_{\operatorname{div}}(\Omega;\C^d)$ is the space of all elements of $\L^2(\Omega;\C^d)$, whose weak divergence is in $\L^2(\Omega)$, and 
$\operatorname{tr}:\H^1(\Omega)\to \H^{1/2}(\partial\Omega)$, $\operatorname{tr}_n:\H_{\operatorname{div}}(\Omega;\C^d)\to \H^{-1/2}(\partial\Omega)$ is the normal trace operator. The proof that this defines a system node is analogous to the arguments in \cite[Sec.~4.1]{PhilippReisSchaller2025}, where Dirichlet control and Neumann observation are considered.

We consider the cost functional
    \begin{equation}
    J(u(\cdot),y(\cdot)) = 2\int_0^\infty 2\re\langle y(\tau),u(\tau)\rangle_{\H^{1/2},\H^{-1/2}}\d\tau,
    \label{eq:CostHeat}
    \end{equation}
where the latter denotes the duality product. Note that this corresponds to a~cost functional of the form \eqref{eq:CostIntro}, in which $Q=0$, $R=0$, and $S:\H^{-1/2}(\partial\Omega)\to \H^{1/2}(\partial\Omega)$ is the inverse of the Riesz isomorphism (which is linear due to the fact that $\H^{-1/2}(\partial\Omega)$ is the set of conjugate linear functionals).

Integration by parts gives
    \begin{multline}\label{eq:ImpPass:Heat}
        \re\langle A\& B\spvek{x_0}{u_0},x_0\rangle_\X = \re\int_\Omega \overline{x_0(\xi)}\Delta x_0(\xi)\d\xi\\ = \re\langle \operatorname{tr}x_0,\operatorname{tr}_n\nabla x_0\rangle_{\H^{1/2},\H^{-1/2}} - \Vert\nabla x_0\Vert^2_{\L^2}
= \re\langle \operatorname y_0,u_0\rangle_{\H^{1/2},\H^{-1/2}} - \Vert\nabla x_0\Vert^2_{\L^2}.    \end{multline}
    which shows that the system is dissipative with respect to the supply rate $s(y,u) = 2\re\langle y,u\rangle_{\H^{1/2},\H^{-1/2}}$ and storage function $\mathcal{S}(x_0) = -\Vert x_0\Vert^2_\X$. Consequently, the operator $P$ in \eqref{eq:RiccatiDissip} is given by $P = -I_\X$. Now we show that Assumption \ref{assum:StateFeedbackStab} is satisfied. We apply the negative output feedback $u(\cdot) = -\mathcal{R}\,\big( C\& D\spvek{x(\cdot)}{u(\cdot)}\big) + v(\cdot)$, where $\mathcal{R}$ is the Riesz isomorphism on $\H^{1/2}(\partial\Omega)$. Then \cite[Thm.~2.5]{HasPau2025} yields that this leads to an infinite-time well-posed system node whose semigroup generator is defined by
    \begin{equation*}
        \label{eq:AK_Heat}
        \begin{aligned}
            A^Kx_0 &= \Delta x_0,\\
            \dom(A^K) &= \left\{x_0\in\H^1(\Omega)\,\vert\,\nabla x_0\in \H_{\operatorname{div}}(\Omega;\C^d), \operatorname{tr}_n \nabla x_0 = -\operatorname{tr} x_0\right\}.
        \end{aligned}
    \end{equation*}
    In order to show that Assumption \ref{assum:StateFeedbackStab} is satisfied, it remains to show that $A^K$ generates a strongly stable $C_0$-semigroup. For this, observe that for $x_0\in\dom(A^K)$, we have
    \begin{equation}
        \langle A^Kx_0,x_0\rangle_{\L^2} = \|\operatorname{tr}x_0\|_{\H^{1/2}}^2 - \Vert\nabla x_0\Vert^2_{\L^2}.
        \label{eq:AK_IP}
    \end{equation}
    By the trace theorem \cite[Thm.~1.5.1.3]{Grisvard1985},
    there exists some $E\in\Lb(\H^{1/2}(\partial\Omega), \H^1(\Omega))$ with  $\operatorname{tr}E = I$. Let $x \in \H^1(\Omega)$. Using 
    \[
x = (x - E\,\operatorname{tr}x) + E\,\operatorname{tr}x,
\]
we have $x - E\,\operatorname{tr}x \in \H^1_0(\Omega)$. By the Poincaré inequality on $\H^1_0(\Omega)$, see \cite[Lem.~10.2]{Tartar2007}, and the boundedness of $E$, we obtain that there exists some $C>0$, such that
\[
\|x\|_{\L^2(\Omega)}^2
\le
C\big(\|\nabla x\|_{\L^2(\Omega)}^2 + \|\operatorname{tr}x\|_{\H^{1/2}(\partial\Omega)}^2\big)\quad\forall\,x \in \H^1(\Omega).
\]
Combining the latter with \eqref{eq:AK_IP}, we have
\[\langle A^Kx_0,x_0\rangle_{\L^2}\leq -C^{-1}\|x_0\|^2_{\L^2}\quad\forall\,x \in \dom(A^K),\]
which further implies that the semigroup generated by $A^K$ is exponentially stable. Altogether, this implies that Assumption \ref{assum:StateFeedbackStab} is satisfied. Now we turn our attention to the determination of the dissipation output $w(\cdot)$ and the Hilbert space $\W$. According to \eqref{eq:ImpPass:Heat}, it holds that
    \begin{align*}
        2\re\langle A\& B\spvek{x_0}{u_0},-I_\X x_0\rangle_\X + s(C\&D\spvek{x_0}{u_0},u_0) = 2\Vert\nabla x_0\Vert_\W^2,
    \end{align*}
    which entails that $K\& L\spvek{x_0}{u_0} = \sqrt{2}\nabla x_0$ and $\W = \L^2(\Omega;\C^d)$. This output is due to internal dissipativity of the heat equation and it highlights that, despite the fact that the input and the output are located at the boundary, the dissipative output can lie in an infinite-dimensional space.

According to Theorem \ref{thm:ReExprCost}, minimizing the cost \eqref{eq:CostHeat} is equivalent in minimizing 
\begin{equation*}
    J_w(u(\cdot),w(\cdot)) = \int_0^\infty \Vert w(\cdot,\tau)\Vert_{\L^2}^2\d\tau,
\end{equation*}
where $w(\cdot) = K\& L\spvek{x(\cdot)}{u(\cdot)} = \sqrt{2}\nabla x(\cdot)$. According to Proposition \ref{prop:ValueFct}, the link between the corresponding value functions is given by $\mathrm{Val}_J(x_0) + \Vert x_0\Vert_\X^2 = \mathrm{Val}_{J_w}(x_0) = \langle x_0,P_wx_0\rangle_\X$ for some $P_w\in\Lb(\X)$. Moreover, Theorem \ref{thm:MainThm} yields the existence of a Hilbert space $\W_w$ and an operator $E\& F:\Lb(\dom(A\& B),\W_w)$ such that the following operator-node Lur'e equation is satisfied 
\begin{equation}
    2\re \langle P_wx_0,A\& B\spvek{x_0}{u_0}\rangle_\X + s(C\& D\spvek{x_0}{u_0},u_0) = \Vert E\& F\spvek{x_0}{u_0}\Vert_{\W_w}^2 + 2\re\langle x_0,A\& B\spvek{x_0}{u_0}\rangle_\X,
    \label{eq:RiccatiHeat}
\end{equation}
for all $\spvek{x_0}{u_0}\in\dom(A\& B)$. By noting that $2\re\langle x_0,A\& B\spvek{x_0}{u_0}\rangle_\X = s(C\& D\spvek{x_0}{u_0},u_0) - \Vert K\& L\spvek{x_0}{u_0}\Vert^2_{\W}$, the Lur'e equation \eqref{eq:RiccatiHeat} reduces to 
\begin{align*}
    2\re\langle P_wx_0,A\& B\spvek{x_0}{u_0}\rangle_\X = \Vert E\& F\spvek{x_0}{u_0}\Vert_{\W_w}^2 - \Vert K\& L\spvek{x_0}{u_0}\Vert_\W^2.
\end{align*}

\section*{Acknowledgments}
Anthony Hastir is supported by the FRS-FNRS (Belgium), under the grant CR
40010909. Timo Reis gratefully acknowledges funding by the Deutsche Forschungsgemeinschaft (DFG,
German Research Foundation) – Project-ID 531152215 – CRC 1701.

\printbibliography

\end{document}